\documentclass[reqno, 12pt]{amsart}

\usepackage{amsthm, amsmath, amssymb, times, enumerate, mathrsfs, enumitem, verbatim, cite}
\usepackage{hyperref}
\hypersetup{colorlinks=false}

\makeatletter
\@namedef{subjclassname@2010}{%
  \textup{2010} Mathematics Subject Classification}
\makeatother

\numberwithin{equation}{section}

\newtheorem{theorem}{Theorem}[section]
\newtheorem{lemma}[theorem]{Lemma}
\newtheorem{proposition}[theorem]{Proposition}

\theoremstyle{definition}
\newtheorem*{ack}{Acknowledgements}
\newtheorem{remark}[theorem]{Remark}

\newenvironment{claim}[1]{\par\noindent\underline{Claim:}\space#1}{}

\renewcommand{\phi}{\varphi}
\renewcommand{\rho}{\varrho}

\renewcommand{\leq}{\leqslant}

\renewcommand{\geq}{\geqslant}

\newcommand{\bs}{\boldsymbol}

\renewcommand{\c}{\mathbf{c}}

\newcommand{\ve}{\varepsilon}

\DeclareMathOperator{\supp}{supp}

\DeclareMathOperator{\arcsinh}{arcsinh}

\newcommand{\sumstar}{\sideset{}{^*}\sum}

\renewcommand{\mod}[1]{\hspace{-2.9mm}\pmod{#1}}

\begin{document}
\baselineskip=17pt
\date{}

\title[On class numbers of imaginary quadratic fields]{On correlations between class numbers of imaginary quadratic fields}

\author{
 V Vinay Kumaraswamy}

\address{School of Mathematics\\
University of Bristol\\ Bristol\\ BS8 1TW
\\ UK}

\email{vinay.visw@gmail.com}

\begin{abstract}
Let $h(-n)$ be the class number of the imaginary quadratic field with fundamental discriminant $-n$. We establish an asymtotic formula for correlations involving $h(-n)$ and $h(-n-l)$, over fundamental discriminants that avoid the congruence class $1\pmod{8}$. Our result is uniform in the shift $l$, and the proof uses an identity of Gauss relating $h(-n)$ to representations of integers as sums of three squares. We also prove analogous results on correlations involving $r_Q(n)$, the number of representations of an integer $n$ by an integral positive definite quadratic form $Q$. 
\end{abstract}

\maketitle

\thispagestyle{empty}

\setcounter{tocdepth}{1}

\section{Introduction} 
Given an arithmetic function $a(n)$, it is a natural problem in analytic number theory to estimate moments of $a(n)$, $\textstyle\sum_{n \leq X} a(n)^k$, and shifted sums of the form 
$\textstyle\sum_{n \leq X} a(n)a(n+l)$. When the $a(n)$ are Fourier coefficients of automorphic forms (the divisor function $d(n)$, for example) information on such correlations can be used to understand properties of their corresponding $L$-functions.

Let $K = \mathbf{Q}(\sqrt{-n})$ be an imaginary quadratic field and $h(-n) = \#Cl_K$ be its class number. In this note we study the shifted sum
\begin{equation}\label{eq:eq1}
D(X,l) = \sideset{}{^\flat}\sum_{1\leq n \leq X}h(-n)h(-n-l),
\end{equation}
where $\flat$ in the above sum denotes restriction to $n$ such that both $-n$ and $-n-l$ are fundamental discriminants, and such that neither is congruent to $1 \pmod{8}$. By the class number formula we have that $h(-n) = n^{1/2+o(1)}$, and as a result we expect that $D(X,l) \asymp X^{\frac{3}{2}}(X+l)^{\frac{1}{2}}$. Using the circle method we show that this holds with a power saving error term.
\begin{theorem}\label{thm1}
Let $l \geq 0$ be an integer, and $D(X,l)$ be as above. Let $$\delta = \begin{cases} 1 &\text{if $l = 0$,} \\
0 &\text{otherwise.}\end{cases}$$ Then there exists a constant $\widehat{\sigma}(l) = \textstyle\prod_{p \leq \infty} \sigma_p(l)$ given in ~\eqref{eq:const11}, such that for all $\ve > 0$ the following asymptotic formula holds,
\begin{equation*} 
D(X,l) = \frac{\widehat{\sigma}(l)}{576}X^{\frac{3}{2}}(X+l)^{\frac{1}{2}} + O_{\ve}\left(X^{\frac{3}{2}-\frac{1}{30}}(X+l)^{\frac{1}{2}+\frac{3+\delta}{180}+\ve}\right).
\end{equation*}
Moreover, $\widehat{\sigma}(l) \neq 0$ whenever $\sigma_2(l) \neq 0$, and $\widehat{\sigma}(l) \ll 1$, for an implied constant that is independent of $l$. 
\end{theorem}
\begin{remark}
Theorem ~\ref{thm1} establishes an asymptotic formula for $D(X,l)$ where the main term exceeds the error term as long as $l \ll X^{2-2\ve}$. By contrast, if the $a(n)$ are normalised Fourier coefficients of cusp forms of integral weight, the asymptotic formula $\textstyle\sum_{n \leq X} a(n)a(n+l) \ll X^{1-\ve}$ holds whenever $l \ll X^{2-\frac{14}{39}}$ (or for $l \ll X^{2-2\ve}$, if one assumes the Ramanujan conjecture, see ~\cite{B}). The relative strength of our result may be explained by the fact that our problem reduces to a problem involving quadratic forms in six variables, whereas, when $a(n) = d(n)$, one has to deal with a quadratic form in four variables.
\end{remark}

In contrast to shifted sums, moments of $h(-n)$ have been studied before; we have the following asymptotic formula,
\begin{equation*}\label{eq:wolke}
\frac{1}{X^{k/2}}\sum_{n \leq X}h(-n)^k = c(k)X + O(X^{1-\theta}),
\end{equation*}
where the sum ranges over fundamental discriminants. For fixed $k$, this is a result due to Wolke ~\cite{W}, who showed that the asymptotic formula holds with $\theta=1/4$. Lavrik ~\cite{L1} showed that one can take $k \ll \sqrt{\log X}$, and finally, Granville and Soundararajan ~\cite{GS}, have shown that the asymptotic formula holds in the wider range $k \ll \log X$. The methods used to prove these results rely on the theory of character sums. 
Wolke expects the true order of the error term in ~\eqref{eq:wolke} to be $\theta=1/2$. For the second moment, excluding fundamental discriminants that lie in the residue class $1 \pmod{8}$, we show this to be true for the weighted analogue of $D(X,l)$. 

We also prove a result analogous to Theorem ~\ref{thm1} for the ``non-split" sum, 
\begin{theorem}\label{thm3}
Let $d$ be a non-negative integer. Set
\begin{equation*}
S(X,d) = \sideset{}{^\flat}\sum_{n \leq X}h(-(n^2+d)), 
\end{equation*}
where the $\flat$ denotes restriction to fundamental discriminants $-(n^2+d)$ that avoid the congruence class $1 \pmod{8}$. Then there exists a constant $\widetilde{\sigma}(d) = \textstyle\prod_{p \leq \infty}\widetilde{\sigma}_p(d)$ given in ~\eqref{eq:constant2} , such that for all $\ve > 0$ we have
\begin{equation*}
S(X,d) = \frac{\widetilde{\sigma}(d)}{24}X(X^2+d)^{\frac{1}{2}} + O_{\ve}(X^{\frac{7}{8}}(X^2+d)^{\frac{37}{72}+\ve}).
\end{equation*}
Moreover, $\widetilde{\sigma}(d) \neq 0$ so long as $\tilde{\sigma}_2(d) \neq 0$.
\end{theorem}

\subsection{Correlations involving $r_Q(n)$}
Let $r(n) = 4\textstyle\sum_{d \mid n}\chi(d)$, where $\chi$ is the unique non-principal real character modulo $4$, be the number of representations of an integer $n$ as a sum of two squares. For odd $l$, Iwaniec ~\cite[Theorem 12.5]{I} showed that 
\begin{equation*}
\sum_{n\leq X} r(n)r(n+l) = 8\left(\sum_{d \mid l}\frac{1}{d}\right)X + O(l^{\frac{1}{3}}X^{\frac{2}{3}}).
\end{equation*}
As a result, the main term dominates the error term when $l \ll X^{1-\ve}$. In our next result we show that the asymptotic formula holds in the wider range, $1\leq l \ll X^{\frac{4}{3}-o(1)}$, and we impose no other restrictions on $l$.
\begin{theorem}\label{iwaniec}
Let $l \geq 1$ be an integer. There exists a constant $c=c(l)$ such that for all $\ve > 0$ we have
\begin{equation*}
\sum_{n\leq X} r(n)r(n+l) = cX+O_{\ve}(X^{\frac{4}{5}}(X+l)^{\frac{3}{20}+\ve}).
\end{equation*}
\end{theorem}

More generally, let $Q$ be an integral positive definite quadratic form and let $n$ be an integer. Let $r_Q(n)$ denote the number of representations of $n$ by $Q$. For instance if $Q=x_1^2+x_2^2$, $r_Q(n) = r(n)$. We establish the following result on correlations between $r_Q(n)$.
\begin{theorem}\label{thmr}
Let $Q_1$ and $Q_2$ be two integral positive-definite quadratic forms in $m \geq 3$ variables. Let $\delta$ be as in the statement of Theorem ~\ref{thm1}. Then there exists a constant $c=c(Q_1,Q_2,l)$ that depends on the quadratic forms $Q_i$ and the shift $l$, such that for all $\ve >0$ we have
\begin{equation*}\begin{split}
\sum_{m \leq X} r_{Q_1}(m)r_{Q_2}(m+l) &= cX^{\frac{m}{2}}(X+l)^{\frac{m}{2}-1} + \\ &\quad \quad O_{\ve}\left(X^{\frac{m}{2}-\frac{m}{2(2m+1)}}(X+l)^{\frac{m}{2}-1+\frac{3+\delta}{4(2m+1)}+\ve}\right).
\end{split}
\end{equation*}
\end{theorem}

The key input in this paper is a result that counts the number of integer points on $Q_1(\bs{x}_1)-Q_2(\bs{x}_2)-l=0$ that satisfy certain congruence conditions - where $Q_i$ are integral, positive definite quadratic forms - which is uniform in $l$ and the congruence conditions. The proof uses Heath-Brown's variant of a certain $\delta$-method ~\cite[Theorem 1]{HB}, first developed by Duke, Friedlander and Iwaniec ~\cite{DFI1}. Although our methods are similar to ~\cite[Theorem 4]{HB}, the main difficulty in our analysis arises in estimating exponential integrals involving lopsided weight functions. Ultimately, our error terms are as good as those in ~\cite{HB}. Another interesting aspect of our result is that it allows us to handle the `split' (Theorem ~\ref{thm1}) and `non-split' (Theorem ~\ref{thm3}) sums simultaneously.

\subsection*{Notation}
By $\|w\|_{N,1}$ we denote the $L^1$ Sobolev norm of order $N$ of a function $w$. All implicit constants that appear in the error terms will be allowed to depend on the underlying quadratic forms. Any further dependence will be indicated by an appropriate subscript. 

\section{The main proposition}
In this section we adopt the convention where a $(k_1+k_2)$-tuple $\bs{x}$ is written $\bs{x} = (\bs{x}_1,\bs{x}_2)$, with $\bs{x}_i$ being $k_i$-tuples. Let $Q_1$ and $Q_2$ be positive-definite integral quadratic forms in $k_1$ and $k_2$ variables, respectively, and let $n = k_1+k_2$. Let $A_1,A_2$ be positive integers, and $\bs{a}_i \in \left(\mathbf{Z}/A_i \mathbf{Z}\right)^{k_i}$ be fixed residue classes. Let $w(\bs{x})$ be a non-negative smooth function with compact support in $\mathbf{R}^n$ such that $\|w\|_{1,N} \ll_N \|w\|_{1,1}^N$, let $l$ be a non-negative integer and set $Y=X+l$. Define the sum
\begin{equation}\label{eq:trunks}
S(\bs{a}_1,\bs{a}_2)=\sum_{\substack{\bs{x} \in \mathbf{Z}^{k_1} : \bs{x} \equiv \bs{a}_1 \mod{A_1} \\ \bs{y} \in \mathbf{Z}^{k_2} : \bs{y} \equiv \bs{a}_2 \mod{A_2} \\ Q_1(\bs{x})-Q_2(\bs{y}) = l}} w\left(\frac{\bs{x}}{Y^{\frac{1}{2}}},\frac{\bs{y}}{X^{\frac{1}{2}}}\right).
\end{equation}
In this section we give an asymptotic formula for $S(\bs{a}_1,\bs{a}_2)$,
\begin{proposition}\label{propmain}
Let $\ve > 0$ and define $$\delta = \begin{cases} 1 &\mbox{$l = 0$ and $n$ is even,} \\
0 &\mbox{otherwise.}\end{cases}$$ Then with notation as above, we have
\begin{equation*}
\begin{split}
S(\bs{a}_1,\bs{a}_2) &= \frac{Y^{\frac{k_1}{2}-1}X^{\frac{k_2}{2}}}{A_1^{k_1}A_2^{k_2}}\left(c_{\infty}(w,l)\prod_{p < \infty}\c_p  + O_{\ve}\left(\|w\|_{1,1}^{\frac{n}{2}}(A_1A_2)^{\frac{7n}{2}}Y^{\frac{3+\delta}{4}+\ve}X^{-\frac{n}{4}}\right)\right) \\ &\quad\quad\quad+ O_{\ve}\left(\|w\|_{1,1}^{n}(A_1A_2)^{3n} Y^{\frac{n-1+\delta}{4}+\ve}\right),
\end{split}
\end{equation*}
where
\begin{equation}\label{eq:cwl}
c_{\infty}(w,l) = \lim_{\kappa \to 0}\frac{1}{2\kappa} \int_{|Q_1(\bs{x}_1) - \frac{XQ_2(x_2)+l}{Q^2}| \leq \kappa} w(\bs{x}) \, d\bs{x},
\end{equation}
and
\begin{equation}\label{eq:ccp}
c_p = \lim_{t \to \infty} \frac{\# \left\{\begin{split}&\boldsymbol{x}_1 \mod{p^{t+v_p(A_1)}} \\ &\boldsymbol{x}_2 \mod{p^{t+v_p(A_2)}} \end{split}: \begin{split} &\boldsymbol{x}_1\equiv \boldsymbol{a}_1 \mod{p^{v_p(A_1)}} \\ &\boldsymbol{x}_2\equiv \boldsymbol{a}_2 \mod{p^{v_p(A_2)}}\\& p^t \mid Q_1(\boldsymbol{x}_1)-Q_2(\boldsymbol{x}_2)-l \end{split} \right\}}{p^{(n-1)t}}.
\end{equation}
\end{proposition}

\begin{remark}
One can also establish Proposition ~\ref{propmain} by adapting the proof of the main theorem in ~\cite{HB2}, which uses the classical Hardy-Littlewood circle method. However, it appears difficult to get a result that is uniform in the shift $l$ in a wide range using this method. 
\end{remark}

\subsection{The $\delta$-method}
Let 
\begin{equation*}\delta(n) = \begin{cases} 1 &\mbox{$n = 0,$} \\ 
0 &\mbox{otherwise.}
\end{cases}\end{equation*}
Heath-Brown ~\cite[Theorem 1]{HB} has established the following decomposition of the $\delta$-symbol in terms of additive characters and the function $h(x,y)$, which closely resembles the Dirac delta at $0$, 
\begin{equation*}
\delta(n) = c_Q Q^{-2}\sum_{q=1}^{\infty} \sumstar_{d \mod{q}}e_q(dn) h\left(\frac{q}{Q},\frac{n}{Q^2}\right),
\end{equation*}
for any $Q > 1$. Using the $\delta$-symbol to detect the equation $Q_1(\bs{x})-Q_2(\bs{y})-l=0$ in ~\eqref{eq:trunks} with $Q^2 = X+l = Y$ we get that (see ~\cite[Theorem 2]{HB})
\begin{equation*}\label{eq:postpoisson}
S(\bs{a}_1,\bs{a}_2) = \frac{c_QY^{\frac{k_1}{2}-1}X^{\frac{k_2}{2}}}{A_1^{k_1}A_2^{k_2}}\sum_{q=1}^{\infty}\frac{1}{q^{n}}\sum_{\bs{c} \in \mathbf{Z}^{n}}S_q(\bs{c})I_q(\bs{c}),
\end{equation*}
where
\begin{equation*}
\begin{split}
S_q(\bs{c}) &= \sumstar_{d \mod{q}}\sum_{\substack{\bs{x}_1 \in \left(\mathbf{Z}/qA_1\mathbf{Z}\right)^{k_1} : \bs{x}_1 \equiv \bs{a}_1 \mod{A_1} \\ \bs{x}_2 \in \left(\mathbf{Z}/qA_2\mathbf{Z}\right)^{k_2} : \bs{x}_2 \equiv \bs{a}_2 \mod{A_2}}}e_q\left(d(Q_1(\bs{x}_1)-Q_2(\bs{x}_2)-l)+\bs{c}.\bs{x}\right), \\
I_q(\bs{c}) &= \int_{\mathbf{R}^{k_1}\times \mathbf{R}^{k_2}} w(\bs{x})h\left(r, \frac{YQ_1(\bs{x}_1)-XQ_2(\bs{x}_2)-l}{Y}\right)e_{qA_1/\sqrt{Y}}(-\bs{c}_1.\bs{x}_1) \times \\ &\quad \quad e_{qA_2/\sqrt{X}}(-\bs{c}_2.\bs{x}_2) \, d\bs{x}. 
\end{split}
\end{equation*}
By properties of $h(x,y)$, only the terms $q \ll Q$ contribute to the above sum. We shall see that the main term in the asymptotic formula for $S(\bs{a}_1,\bs{a}_2)$ comes from $\bs{c} = \bs{0}$, and we now turn to bounding the exponential sums and integrals.
\subsection{Analysis of the exponential sum}
The following is a straightforward consequence of ~\cite[Lemma 23]{HB}.
\begin{lemma}\label{multlem}
Let $q=q'q''$, $A_1=A_1'A_1''$ and $A_2=A_2'A_2''$ such that $(A_1'q',A_1''q'')=(A_2'q',A_2''q'')=1$. Then we have
\begin{equation*}
S_q(\bs{c}) = S_{q'}(\overline{q''}\bs{c})S_{q''}(\overline{q'}\bs{c}),
\end{equation*}
where $\overline{q'}q' \equiv 1 \pmod{q''}$ and $\overline{q''}q'' \equiv 1 \pmod{q'}$. 
\end{lemma}
Next, we give a preliminary bound for $S_q(\bs{c})$ which is analogous to ~\cite[Lemma 25]{HB}.
\begin{lemma}\label{lemma25}
We have
\begin{equation*}
S_q(\bs{c}) \ll (A_1A_2)^nq^{1+\frac{n}{2}}.
\end{equation*}
\begin{proof}
Let $F(\bs{x}) = F^{0}(\bs{x}) - l = Q_1(\bs{x}_1)-Q_2(\bs{x}_2) - l$. By Cauchy's inequality,
\begin{equation*}
\begin{split}
|S_q(\bs{c})| &\leq (\phi(q))^{\frac{1}{2}}\widetilde{S}^{\frac{1}{2}}_q(\bs{c}) ,
\end{split}
\end{equation*}
where
\begin{equation*}
\widetilde{S}_q(\bs{c}) = \sideset{}{^{*}} \sum_{d \mod{q}}\bigg|\sum_{\substack{\bs{x}_1,\bs{y}_1 \mod{A_1q}: \bs{x}_1,\bs{y}_1 \equiv \bs{a}_1 \mod{A_1} \\ \bs{x}_2,\bs{y}_2 \mod{A_2q}: \bs{x}_2,\bs{y}_2 \equiv \bs{a}_2 \mod{A_2}}}e_q(d(F(\bs{x})-F(\bs{y})) + \bs{c}.(\bs{x}-\bs{y}))\bigg|.
\end{equation*}
Set $\bs{x}-\bs{y} = \bs{z}$. Then $\bs{z}_1 \equiv 0 \pmod{A_1}$ and $\bs{z_2} \equiv 0 \pmod{A_2}$. Furthermore, $F(\bs{x})-F(\bs{y}) = F^{0}(\bs{z}) + \nabla F(\bs{z}).\bs{y}.$ 
\begin{equation*}
\sum_{\bs{z}} e_q(dF^{0}(\bs{z}) + \bs{c}.\bs{z})\sum_{\substack{\bs{y}_1\mod{A_1q}: \bs{y}_1 \equiv \bs{a}_1 \mod{A_1} \\ \bs{y}_2\mod{A_2q}: \bs{y}_2 \equiv \bs{a}_2 \mod{A_2}}} e_q(d\bs{y}.\nabla F(\bs{z})).
\end{equation*}
 If $M$ is the matrix representing the quadratic form $F_0$, then $\nabla F(\bs{z}) = 2M\bs{z}$. The sum over $\bs{y}$ above is $0$ unless $M\bs{z_1} \equiv 0 \pmod{q/(q,A_1)}$ and \newline $M\bs{z}_2 \equiv 0 \pmod{q/(q,A_2)}$. Since this happens for only $O((q,A_1^2)^{k_1}(q,A_2^2)^{k_2})$ of the $\bs{z}$, we have $$\widetilde{S}_q(\bs{c}) \leq (q,A_1^2)^{k_1}(q,A_2^2)^{k_2}q^{1+n}.$$ This completes the proof of the lemma.
\end{proof}
\end{lemma}
The following is the key result on exponential sums that we shall need, and it is similar to ~\cite[Lemma 28]{HB}.
\begin{lemma}\label{lemma28}
Let $$\delta = \begin{cases} 1 &\mbox{$l = 0$ and $n$ is even,} \\
0 &\mbox{otherwise.}\end{cases}$$
\begin{equation}\label{eq:lnot0}
\sum_{q \leq Z} |S_q(\bs{c})| \ll (A_1A_2)^{2n}Z^{\frac{3+n+\delta}{2}+\ve}X^{\ve},
\end{equation}
\end{lemma}
\begin{proof}
When $l=0$ and $n$ is even, the result follows from Lemma ~\ref{lemma25}. When $l \neq 0$, or $n$ is odd, we factorise $q=u_1u_2v$ such that $(u_1,u_2)=1$, $u_1$ and $u_2$ are square-free and $v$ is square-full and $(u_1u_2,v) = 1$,  $(u_1,A_1A_2)=1$. Then we have by Lemma ~\ref{multlem} that
\begin{equation*}
S_q(\bs{c}) = S_{u_1}(\overline{u_2v}\bs{c})S_{u_2}(\overline{u_1v}\bs{c})S_v(\overline{u_1u_2}\bs{c}).
\end{equation*}
Let $M$ be the matrix that corresponds to the quadratic form $F^{0}(\bs{x}) = Q_1(\bs{x}_1)-Q_2(\bs{x}_2)$. Let $M^{-1}(\bs{x})$ denote the quadratic form that corresponds to the matrix $M^{-1}$, which is well-defined modulo $p$ if $p \nmid \det M$. We use the bounds,
\begin{equation*}
\begin{split}
S_{u_1}(\overline{u_2v}\bs{c}) &\ll C^{\omega(u_1)}u_1^{\frac{1+n}{2}}(l,M^{-1}(\bs{c}),u_1)^{\frac{1}{2}}, \\
S_v(\overline{u_1u_2}\bs{c}) &\ll A_1^{k_1}A_2^{k_2}v^{1+\frac{n}{2}},  \\
\end{split}
\end{equation*}
and
\begin{equation*}
\begin{split}
S_{u_2}(\overline{u_1v}\bs{c}) &\ll  A_1^{\frac{k_1}{2}}A_2^{\frac{k_2}{2}}(A_1A_2)^{\frac{1}{4}}u_2^{\frac{1+n}{2}}.
\end{split}
\end{equation*}
The first bound follows from ~\cite[Lemma 26]{HB}, where $C$ is a constant that depends only on $\det M$. The last bound follows from the trivial bound, $S_p(\bs{c}) \ll p^{1+n}$, and by noticing that if $p \mid u_2$ then $p \mid A_1A_2$. Inserting these bounds into the proof of ~\cite[Lemma 28]{HB} we obtain ~\eqref{eq:lnot0}. 
\end{proof}

\subsection{Estimates for exponential integrals I}
Let
\begin{equation}\label{eq:testtest}
w_0(x) = \begin{cases} \exp(-(1-x^2)^{-1}), &\mbox{$|x| < 1$} \\
0 &\mbox{$|x| \geq 1,$}
\end{cases}
\end{equation}
and let $\omega(x) = w_0\left(\frac{x}{6n(\|Q_1\|+\|Q_2\|)}\right)$, where $\|Q\|$ denotes the norm of a quadratic form $Q$, which is the largest coefficient of $Q$ in absolute value. Let $$z(\bs{x}) = \frac{YQ_1(\bs{x}_1)-XQ_2(\bs{x}_2)-l}{Y}.$$ Then $\omega(z(\bs{x})) \gg 1$ whenever $\bs{x} \in \supp(w)$. We have
\begin{equation*}
I_q(\bs{c}) = \int_{\mathbf{R}^{n}} w_3(\bs{x}) f(z(\bs{x}))e(-\bs{u}.\bs{x}) \, d\bs{x},
\end{equation*} 
where $f(y) = h(r,y)\omega(y),$ and $w_3(\bs{x}) = \frac{w(\bs{x})}{\omega(z(\bs{x}))}.$ Observe that $f(y)$ has compact support. Let $r=q/Q$, then by ~\cite[Lemma 17]{HB} we have the following bound for the Fourier transform of $f$,
\begin{equation}\label{eq:prt}
\begin{split}
p(t) &= p_r(t) = \int_{\mathbf{R}} f(r,y)\omega(y) e(-ty) \, dy \ll_j (r|t|)^{-j}.
\end{split}
\end{equation}
This bound shows that $p(t)$ has polynomial growth in $r$ (recall that $r \ll 1$) if $|t| \gg r^{-1-\ve}$. 

Let $\bs{u}_1 = \frac{\bs{c}_1}{qA_1/\sqrt{Y}}$, and $\bs{u}_2 =  \frac{\bs{c}_2}{qA_2/\sqrt{X}}$.  By Fourier inversion we see that
\begin{equation}\label{eq:iutt}
\begin{split}
I_q(\bs{c}) &= \int_{\mathbf{R}} p(t)e(-tl/Y) I(\bs{u},t) dt,
\end{split}
\end{equation}
where
\begin{equation*}
I(\bs{u},t) = \int_{\mathbf{R}^{n}} w_3(\bs{x}) e\left(tQ_1(\bs{x}_1)-t\frac{X}{Y}Q_2(\bs{x}_2) - \bs{u}.\bs{x}\right) \, d\boldsymbol{x}.
\end{equation*}
The key result in this section is
\begin{lemma}\label{restrictrangeprop}
Let $\ve > 0$ be fixed and let $\bs{c} \neq \bs{0}$. Assume that $\|w\|_{1,1}\gg 1$ and that $\|w\|_{N,1} \ll_N \|w\|_{1,1}^N.$ Then we have 
\begin{equation*}
I_q(\boldsymbol{c}) \ll_A X^{-A},
\end{equation*}
if $\frac{|\boldsymbol{c_1}|}{A_1} \gg \|w\|_{1,1}X^{\ve}$ or $\frac{|\boldsymbol{c_2}|\sqrt{X}}{\sqrt{Y}A_2} \gg \|w\|_{1,1}X^{\ve}$. 
\end{lemma}

\begin{proof}
Modifying the proof of ~\cite[Lemma 10]{HB} to keep track of any dependency on $w$, we have for $M > 0$ that
\begin{equation}\label{eq:lemma10hb}
I(\bs{u},t) \, \ll_{M} \|w\|_{M,1}|\bs{u}|^{-M},
\end{equation}
when $|t| \ll |\bs{u}|$. Using ~\eqref{eq:prt} when $|t| \gg |\bs{u}|$, we get by ~\eqref{eq:iutt} that
\begin{equation*}
I_q(\bs{c}) \ll_M \|w\|_{M,1}r^{-1}|\bs{u}|^{-M} + r^{-M}|\bs{u}|^{1-M}.
\end{equation*}
If $|\bs{u}| \gg r^{-1}\|w\|_{1,1}X^{\ve}$, we see that
\begin{equation*}
I_q(\bs{c}) \ll_A X^{-A}.
\end{equation*}
This completes the proof.
\end{proof}
\subsection{Estimates for exponential integrals II}
By Proposition ~\ref{restrictrangeprop} we have arbitrary polynomial decay for $I_q(\bs{c})$ unless $|\bs{u}_1|\leq r^{-1}\|w\|_{1,1}X^{\ve}$, and $|\bs{u}_2| \ll r^{-1}\|w\|_{1,1}X^{\ve}.$ To estimate the integral in this range, we need the following 
\begin{lemma}\label{smalluprop}
Let $\ve > 0$ and $\bs{c} \neq \bs{0}$. If $\bs{c}_1=\bs{0}$. Then we have
\begin{equation}\label{eq:c10}
I_q(\bs{c}) \ll (r^{-1}|\bs{u}_2|)^{\ve}\|w\|_{1,1}^{\frac{n}{2}}r^{-1}|\bs{u}_2|^{-\frac{n}{2}}.
\end{equation}
Suppose $\bs{c}_2 = \bs{0}$, and let $\ve > 0$. Then we have
\begin{equation}\label{eq:c20}
I_q(\bs{c}) \ll (r^{-1}|\bs{u}_1|)^{\ve}\|w\|_{1,1}^{\frac{n}{2}}r^{-1}|\bs{u}_1|^{-\frac{n}{2}}.
\end{equation}
If $\bs{c}_1, \bs{c}_2 \neq \bs{0}$, we have
\begin{equation}\label{eq:cnot0}
I_q(\bs{c}) \ll (r^{-1}|\bs{u}|)^{\ve}\|w\|_{1,1}^{\frac{n}{2}}r^{-1}|\bs{u}_1|^{-\frac{k_1}{2}}|\bs{u}_2|^{-\frac{k_2}{2}} 
\end{equation}
\end{lemma}
\begin{proof}
We begin by recording the trivial bound, $I_q(\bs{c}) \ll 1$, which follows from ~\cite[Lemma 15]{HB}. Next, using the fact that $w$ is compactly supported, we may write 
\begin{equation*}
\begin{split}
I(\bs{u},t) &= \int_{\mathbf{R}^{k_1}}e(tQ_1(\bs{x}_1)-\bs{u}_1.\bs{x_1})\times \\ &\quad \quad  \int_{\mathbf{R}^{k_2}}w_3(\bs{x})e\left(-t\frac{X}{Y}Q_2(\bs{x}_2)-\bs{u}_2.\bs{x}_2\right) \, d\bs{x}_2 \, d\bs{x}_1.
\end{split}
\end{equation*} 
Integrating trivially over $\bs{x}_1$, and estimating the integral over $\bs{x}_2$ by ~\eqref{eq:lemma10hb}, we get that 
\begin{equation*}\label{eq:x1first}
I(\bs{u},t) \ll_N \|w\|_{N,1}|\bs{u}_2|^{-N}
\end{equation*} 
if $|t| \ll \frac{Y}{X}|\bs{u}_2|$. Arguing similarly with the roles of $\bs{x}_1$ and $\bs{x_2}$ interchanged, we also have the bound
\begin{equation*}\label{eq:x2 first}
I(\bs{u},t) \ll_N \|w\|_{N,1}|\bs{u}_1|^{-N}
\end{equation*}
if $|t| \ll |\bs{u}_1|$. Finally, we record the following bound from ~\cite[Lemma 3.1]{HBP},
\begin{equation}\label{eq:hbp}
I(\bs{u},t) \ll \begin{cases} |w\|_{N,1}|\bs{u}|^{-N}\ &\mbox{$|\bs{u}| \gg |t|$}\\
\prod_{i=1}^{k_1}\min\left(1,\left(|t|\right)^{-\frac{1}{2}}\right)\prod_{j=1}^{k_2}\min\left(1,\left(|t|\frac{X}{Y}\right)^{-\frac{1}{2}}\right) &\mbox{$\forall t \in \mathbf{R}.$}
\end{cases}
\end{equation}
In addition to the dependence on the quadratic forms $Q_i$, the implied constant for the first bound depends on $N$, and for the second bound the dependence is also on the $L^{1}$ norm of the weight function $w$. The above bounds are sufficient to prove the lemma. We also remark that since $w$ is assumed to be compactly supported away from the origin, we see that $\|w\|_{1,1} \gg 1$. 

Suppose first that $\bs{u}_1=\bs{0}$. We have by ~\eqref{eq:iutt} and ~\eqref{eq:hbp} that
\begin{equation}\label{eq:manavinala}
\begin{split}
I_q(\bs{c}) &\ll \int_{|t| \ll \frac{Y}{X}|\bs{u}_2|}|p(t)|\|w\|_{N,1}|\bs{u}_2|^{-N} \, dt + \int_{|t| \gg \frac{Y}{X}|\bs{u}_2|} |p(t)|\left(\frac{X}{Y}\right)^{\frac{k_1}{2}}|\bs{u}_2|^{-\frac{n}{2}} \, dt \\
&\ll r^{-1}\|w\|_{N,1}|\bs{u}_2|^{-N} + r^{-1}\left(\frac{X}{Y}\right)^{\frac{k_1}{2}}|\bs{u}_2|^{-\frac{n}{2}}.
\end{split}
\end{equation} 
Here we have used the fact that $\textstyle\int_{-\infty}^{\infty}|p(t)| \, dt \ll r^{-1}$. If $|\bs{u}_2| \gg r^{-2\ve/n}\|w\|_{1,1}$, using the fact that $\|w\|_{N,1} \ll \|w\|_{1,1}^N$, and by choosing $N$ large enough we get that $$r^{-1}\|w\|_{N,1} |\bs{u}_2|^{-N} \ll_N r^{-1+N\ve} \ll r^{-1}|\bs{u}_2|^{-\frac{n}{2}}.$$ If $|\bs{u}_2| \ll r^{-2\ve/n}\|w\|_{1,1}$, observe that $$|\bs{u}_1|^{\frac{n}{2}-\ve} \ll \|w\|_{1,1}^{\frac{n}{2}}r^{-\ve}.$$ As a result, we have that $$1 \ll (r^{-1}|\bs{u}_2|)^{\ve}\|w\|_{1,1}^{\frac{n}{2}}r^{-1}|\bs{u}_2|^{-\frac{n}{2}}.$$ Since $$I_q(\bs{c}) \ll 1 \ll (r^{-1}|\bs{u}_2|)^{\ve}\|w\|_{1,1}^{\frac{n}{2}}r^{-1}|\bs{u}_2|^{-\frac{n}{2}},$$ this completes the proof of ~\eqref{eq:c10}. The proof of ~\eqref{eq:c20} follows from an analogous argument, replacing $\bs{u}_2$ by $\bs{u}_1$.

Finally, consider the case when $|\bs{u}_1|$ and $|\bs{u}_2|$ are both non-zero. The proof of ~\eqref{eq:cnot0} follows from combining ~\eqref{eq:c10} and ~\eqref{eq:c20} - observe first that these bounds hold even if $\bs{c}_1 \neq \bs{0}$, $\bs{c}_2 \neq \bs{0}$, respectively. If $|\bs{u}_1| \ll |\bs{u}_2|$, we use ~\eqref{eq:c10} and the fact that $|\bs{u}_2|^{-\frac{k_1}{2}} \ll |\bs{u}_1|^{-\frac{k_1}{2}}.$ If $|\bs{u}_2| \ll |\bs{u}_1|$ we use ~\eqref{eq:c20}, and this completes the proof of the lemma.
\end{proof}
\subsection{Evaluating $I_q(\bs{0})$}
Recall that
\begin{equation*}
I_q(\bs{0}) = \int_{\mathbf{R}^{n}} w(\bs{x})h\left(r, Q_1(\bs{x}_1)-\frac{XQ_2(\bs{x}_2)+l}{Y}\right) d\bs{x}.
\end{equation*} 
We show that the following holds,
\begin{lemma}\label{infty}
If $q \ll Q$, we have for all $N \geq 1$ that
\begin{equation*}
\begin{split}
I_q(\boldsymbol{0}) &= c_{\infty}(w,l) + O_N(\|w\|_{N,1}r^N),
\end{split}
\end{equation*}
where
$$c_{\infty}(w,l) = \lim_{\kappa \to 0} \frac{1}{2\kappa} \int_{|Q_1(\boldsymbol{x})-\frac{XQ_2(\boldsymbol{y})+l}{Q^2}|\leq \kappa}w(\boldsymbol{x}) \, d\boldsymbol{x}.$$
\end{lemma}
\begin{proof}
We follow the proof of ~\cite[Lemma 13]{HB}, and also keep track of any dependency on $w$. Let $c_0 = \textstyle\int_{-\infty}^{\infty}w_0(x)\,dx$, where $w_0(x)$ is defined in ~\eqref{eq:testtest}. For $\delta > 0$ define the function \begin{equation*}
w_1(\bs{x}) = w_{\delta}\left(\frac{\bs{x}-\bs{y}}{\delta},\bs{y}\right) = c_0^{-n}\prod_{i=1}^{n}w_0\left(\frac{x_i-y_i}{\delta}\right)w(\bs{x}).
\end{equation*}
Then by ~\cite[Lemmas 9,12,13]{HB} we have that 
\begin{equation*}
\begin{split}
I_q(\bs{0}) &= \delta^{-{n}} \int w_1(\boldsymbol{x}) h\left(r,Q_1(\boldsymbol{x}) - \frac{XQ_2(\boldsymbol{y}) + l}{Q^2}\right) \, d\boldsymbol{x} \, d\boldsymbol{y} \\
&= \delta^{-n} \int \left\{ \lim_{\kappa \to 0} \frac{1}{2\kappa} \int_{|Q_1(\boldsymbol{x_1}) - \frac{XQ_2(\boldsymbol{y}) + l}{Q^2}| \leq \kappa}w_1(\boldsymbol{x}) \, d\bs{x} +   O_N(r^N\|w\|_{N,1})\right\} \; d\bs{y} \\
&=  \lim_{\kappa \to 0} \frac{1}{2\kappa} \int_{|Q_1(\boldsymbol{x})-\frac{XQ_2(\boldsymbol{y})+l}{Q^2}|\leq \kappa}w(\boldsymbol{x}) \, d\boldsymbol{x}  + O_N(\delta^{-n}r^N\|w\|_{N,1}),
\end{split}
\end{equation*}
since $w(\bs{x}) = \delta^{-n} \textstyle \int w_{\delta}\left(\frac{\bs{x}-\bs{y}}{\delta},\bs{y}\right) \, d\bs{y}.$ This completes the proof of the lemma.
\end{proof}

\subsection{Proof of Proposition ~\ref{propmain}}
By Lemma ~\ref{restrictrangeprop}, and the fact that $c_Q = 1+O_A(Q^{-A})$, we get that
\begin{equation*}\label{eq:sab}
\begin{split}
S(\bs{a}_1,\boldsymbol{a}_2) &= \frac{Y^{k_1/2-1}X^{k_2/2}}{A_1^{k_1}A_2^{k_2}}\sum_{\substack{|\boldsymbol{c_1}| \ll \|w\|_{1,1}A_1X^{\varepsilon} \\ |\boldsymbol{c_2}| \ll \|w\|_{1,1}A_2\frac{\sqrt{Y}}{\sqrt{X}}X^{\ve}}}\sum_{q \ll Q} \frac{1}{q^{n}}S_{q}(\bs{c})I_{q}(\boldsymbol{c}) + O_N(Q^{-N}). \\
\end{split}
\end{equation*}
Define the following subsets of $\mathbf{Z}^n$. Let $\mathcal{C}_1 = \left\{\bs{0}\right\}$, \begin{equation*}
\begin{split}
\mathcal{C}_2 &= \left\{\bs{c} \in \mathbf{Z}^n : \begin{split} & \bs{c}_1 = \bs{0}, 1 \leq |\boldsymbol{c_2}| \ll \|w\|_{1,1}A_2\frac{\sqrt{Y}}{\sqrt{X}}X^{\ve} \end{split} \right\} \\
\mathcal{C}_3 &= \left\{\bs{c} \in \mathbf{Z}^n : \begin{split} & 1\leq|\boldsymbol{c_1}| \ll \|w\|_{1,1}A_1X^{\ve}, \boldsymbol{c_2}= \bs{0} \end{split} \right\}, \\
\end{split}
\end{equation*}
and $$\mathcal{C}_4 = \left\{\bs{c} \in \mathbf{Z}^n : \begin{split} &1\leq |\boldsymbol{c_1}| \ll \|w\|_{1,1}A_1X^{\ve} \\ &1\leq |\boldsymbol{c_2}| \ll \|w\|_{1,1}A_2\frac{\sqrt{Y}}{\sqrt{X}}X^{\ve} \end{split} \right\}.$$ We then have
\begin{equation}
\begin{split}
S(\bs{a}_1,\boldsymbol{a}_2) &= \frac{Y^{k_1/2-1}X^{k_2/2}}{A_1^{k_1}A_2^{k_2}}\sum_{i=1}^4\sum_{\substack{\bs{c} \in \mathcal{C}_i}}\sum_{q \ll Q}\frac{1}{q^{n}}S_{q}(\boldsymbol{c})I_{q}(\boldsymbol{c})  + O_N(Q^{-N}) \\ 
&= S_1 + S_2 + S_3 + S_4,
\end{split}
\end{equation}
say. 

\subsubsection{Analysis of the main term}
Using the trivial bound, $I_q(\bs{0}) \ll 1$ we have by Lemma ~\ref{lemma28} that
\begin{equation*}
\sum_{q \sim R} \frac{1}{q^{n}}S_{q}(\boldsymbol{0})I_{q}(\boldsymbol{0}) \ll (A_1A_2)^{2n}R^{\frac{3+\delta}{2}}R^{-\frac{n}{2}}.
\end{equation*}
Hence the terms $q > Q/\|w\|_{1,1}X^{\ve}$ in $S_1$ make a contribution that is $$O_{\ve}\left((A_1A_2)^{2n}\|w\|_{1,1}^{\frac{n}{2}}Y^{\frac{n-1+\delta}{4}+\ve}\right).$$ By Lemma ~\ref{infty} and ~\cite[Lemma 31]{HB} we have
\begin{equation}\label{eq:mt}
\begin{split}
S_1 =& \frac{c_{\infty}(w,l)Y^{\frac{k_1}{2}-1}X^{\frac{k_2}{2}}}{A_1^{k_1}A_2^{k_2}} \sum_{q=1}^{\infty} \frac{1}{q^{n}}S_{q}(\boldsymbol{0}) + O\left((A_1A_2)^{2n}\|w\|_{1,1}^{\frac{n}{2}}Y^{\frac{n-1+\delta}{4}+\ve}\right). \\
\end{split}
\end{equation}
\subsubsection{The leading constant}
Since $S_q(\bs{c})$ is multiplicative, it is a standard computation to show that
\begin{equation*}
\sum_{q=1}^{\infty} q^{-n}S_q(\bs{0}) = \prod s_p,
\end{equation*}
where
\begin{equation*}
s_p = \lim_{t \to \infty} \frac{\# \left\{\begin{split}&\boldsymbol{x_1} \mod{p^{t+v_p(A_1)}} \\ &\boldsymbol{x_2} \mod{p^{t+v_p(A_2)}} \end{split}: \begin{split} &\boldsymbol{x_1}\equiv \boldsymbol{a_1} \mod{p^{v_p(A_1)}} \\ &\boldsymbol{x_2}\equiv \boldsymbol{a}_2 \mod{p^{v_p(A_2)}}\\& p^t \mid Q_1(\boldsymbol{x_1})-Q_2(\boldsymbol{x_2})-l \end{split} \right\}}{p^{(n-1)t}}.
\end{equation*}
\subsubsection{Analysis of the error terms}
Recall that
\begin{equation*}
\begin{split}
S_2 &= \frac{Y^{\frac{k_1}{2}-1}X^{\frac{k_2}{2}}}{A_1^{k_1}A_2^{k_2}}\sum_{\bs{c} \in \mathcal{C}_2}\sum_{q \ll Q}\frac{1}{q^{n}}S_{q}(\boldsymbol{c})I_{q}(\boldsymbol{c}). 
\end{split}
\end{equation*}
By ~\eqref{eq:c20} we have that $$I_q(\bs{c}) \ll \frac{\|w\|_{1,1}^{\frac{n}{2}}A_2^{\frac{n}{2}}Y^{\frac{1}{2}+\ve}q^{\frac{n}{2}-1}}{X^{\frac{n}{4}}|\bs{c}_2|^{\frac{n}{2}}}.$$ As a result, 
\begin{equation*}
\begin{split}
S_2 &\ll \frac{Y^{\frac{k_1}{2}-1}X^{\frac{k_2}{2}}Y^{\frac{1}{2}+\ve}X^{-\frac{n}{4}}\|w\|_{1,1}^{\frac{n}{2}}}{A_1^{k_1}A_2^{\frac{k_2-k_1}{2}}}\sum_{\bs{c} \in \mathcal{C}_2}\frac{1}{|\bs{c}|^{\frac{n}{2}}}\sum_{q \ll \sqrt{Y}}\frac{|S_q(\bs{c})|}{q^{\frac{n}{2}+1}}\\
\end{split}
\end{equation*}
By Lemma ~\ref{lemma28} we obtain the bound, 
\begin{equation*}
\begin{split}
S_2 &\ll (A_1A_2)^{3n}Y^{\frac{k_1}{2}-1}X^{\frac{k_2}{2}}Y^{\frac{3+\delta}{4}+\ve}X^{-\frac{n}{4}}\|w\|_{1,1}^{\frac{n}{2}}\sum_{\bs{c} \in \mathcal{C}_2}\frac{1}{|\bs{c}|^{\frac{n}{2}}}.
\end{split}
\end{equation*}
To handle the sum over $\bs{c}$ we use the fact that
\begin{equation*}
\sum_{\substack{\bs{c} \in \mathbf{Z}^d\\ |\bs{c}| \leq T}} |\bs{c}|^{l} \ll \sum_{t \leq T} t^{d+l-1} \ll 1 + T^{d+l}. 
\end{equation*}
Hence we get
\begin{equation*}
\begin{split}
S_2 &\ll_{\ve}  \|w\|_{1,1}^{\frac{n}{2}}(A_1A_2)^{3n}Y^{\frac{k_1}{2}-1}X^{\frac{k_2}{2}}\left(Y^{\frac{3+\delta}{4}+\ve}X^{-\frac{n}{4}}\right) + \\
&\quad \quad \quad \quad\|w\|_{1,1}^{k_2}(A_1A_2)^{3n}Y^{\frac{n-1+\delta}{4}+\ve}.
\end{split}
\end{equation*}

Next, we consider 
\begin{equation*}
\begin{split}
S_3 &= \frac{Y^{\frac{k_1}{2}-1}X^{\frac{k_2}{2}}}{A_1^{k_1}A_2^{k_2}}\sum_{\bs{c} \in \mathcal{C}_3}\sum_{q \ll Q}\frac{1}{q^{n}}S_{q}(\boldsymbol{c})I_{q}(\boldsymbol{c}). 
\end{split}
\end{equation*}
By ~\eqref{eq:c20} we have
\begin{equation*}
\begin{split}
S_3 &\ll \frac{Y^{\frac{k_1}{4}-\frac{1}{2}+\ve}X^{\frac{k_2}{2}}Y^{-\frac{k_2}{4}}A_1^{\frac{n}{2}}\|w\|_{1,1}^{\frac{n}{2}}}{A_1^{k_1}A_2^{k_2}}\
\sum_{\bs{c} \in \mathcal{C}_3}\frac{1}{|\bs{c}|^{\frac{n}{2}}}\sum_{q \ll \sqrt{Y}}\frac{|S_q(\bs{c})|}{q^{\frac{n}{2}+1}},
\end{split}
\end{equation*}
and proceeding as before, by Lemma ~\ref{lemma28}, and summing over $\mathcal{C}_3$ we get,
\begin{equation*}
\begin{split}
S_3 &\ll_{\ve} \|w\|_{1,1}^{n}(A_1A_2)^{3n}Y^{\frac{n-1+\delta}{4}+\ve}.
\end{split}
\end{equation*}

Finally we have,  
\begin{equation*}
\begin{split}
S_4  = \frac{Y^{\frac{k_1}{2}-1}X^{\frac{k_2}{2}}}{A_1^{k_1}A_2^{k_2}}\sum_{\bs{c} \in \mathcal{C}_4}\sum_{q \ll Q}\frac{1}{q^{n}}S_{q}(\boldsymbol{c})I_{q}(\boldsymbol{c}).
\end{split}
\end{equation*}
Using ~\eqref{eq:cnot0} we see that
\begin{equation*}
\begin{split}
S_4 &\ll \frac{Y^{\frac{k_1}{4}-\frac{1}{2}+\ve}X^{\frac{k_2}{4}}A_1^{\frac{k_1}{2}}A_2^{\frac{k_2}{2}}\|w\|_{1,1}^{\frac{n}{2}}}{A_1^{k_1}A_2^{k_2}}\
\sum_{\bs{c} \in \mathcal{C}_4}\frac{1}{|\bs{c_1}|^{\frac{k_1}{2}}|\bs{c}_2|^{\frac{k_2}{2}}}\sum_{q \ll \sqrt{Y}}\frac{|S_q(\bs{c})|}{q^{\frac{n}{2}+1}}.
\end{split}
\end{equation*}
Using Lemma ~\ref{lemma28} once again to estimate the sum over $q$ and summing over $\bs{c} \in \mathcal{C}_4$, we get that
\begin{equation*}
S_4 \ll_{\ve} \|w\|_{1,1}^{n}(A_1A_2)^{3n}Y^{\frac{n-1+\delta}{4}+\ve}.
\end{equation*}
This completes the proof of Proposition ~\ref{propmain}.
\section{Proof of the main theorems}
We begin by proving Theorem ~\ref{thm1}.  
\subsection{Proof of Theorem ~\ref{thm1}}
First we show that it is sufficient to work with a smoothed version of $D(X,l)$. Let $1 \leq P \leq X$ be a parameter that we will choose later, and let $\alpha$ and $\beta$ be smooth functions with compact support, taking values in $[0,1]$ satisfying $\alpha^{(j)}(x) \ll_j 1$, and $\beta^{(j)}(x) \ll_j P^{j}$ such that
\begin{equation*}\label{eq:w1}
\alpha(x)=\begin{cases} 0 &\mbox{if $x \leq 0$}\\
1 &\mbox{if $1/P \leq x \leq 1$} \\ 
0 &\mbox{if $x \geq 2,$} 
\end{cases}
\end{equation*}
and
\begin{equation*}\label{eq:w2}
\beta(x)=\begin{cases} 0 &\mbox{if $x \leq 0$}\\
1 &\mbox{if $1/P \leq x \leq 1$} \\ 0 & \mbox{if $x \geq 1 + 1/P$}. \end{cases}
\end{equation*}
Define the sum
\begin{equation}\label{eq:smoothsumdefinition}
\widetilde{D}(X,l) = \sideset{}{^\flat}\sum_{m-n=l}h(-m)h(-n)\alpha\left(\frac{m}{X+l}\right)\beta\left(\frac{n}{X}\right).
\end{equation}
Then we have 
\begin{lemma}\label{smoothinglemma}
With notation as above, we have for all $\ve > 0$ that
\begin{equation*}
D(X,l) - \widetilde{D}(X,l) \ll X^{3/2}(X+l)^{1/2+\ve}/P.
\end{equation*}
\end{lemma}
\begin{proof}
By the definition of the smooth weights, 
\begin{equation*}
\begin{split}
\widetilde{D}(X,l) &= \sideset{}{^\flat}\sum_{n \leq X} h(-n)h(-n-l) \\ &\quad + \sideset{}{^\flat} \sum_{X < n \leq X+X/P}h(-n)h(-n-l)\alpha\left(\frac{n+l}{X+l}\right)\beta\left(\frac{n}{X}\right) \\ &\quad +O( \sideset{}{^\flat}\sum_{n < X/P} h(-n)h(-n-l))\\
&=D(X,l) + O(X^{1/2}(X+l)^{1/2+\ve}X/P).
\end{split}
\end{equation*}
\end{proof}
\subsection{Reduction to a counting problem}
Let $r_3(n)$ be the number of representations of $n$ as a sum of three squares. The key idea is to use an identity due to Gauss (see ~\cite[Proposition 5.3.10]{C}),
\begin{equation}\label{eq:r3class}
r_3(n) = 12\left(1-\left(\frac{-n}{2}\right)\right)h(-n),
\end{equation}
which holds when $n < -3$ is a fundamental discriminant. The identity enables us to transform the shifted sum $\sideset{}{^\flat}\sum h(-n)h(-n-l)$ to sums of the form \linebreak$\textstyle\sum r_3(n)r_3(n+l)$, which in turn reduces to the problem of counting integer points in bounded regions that lie on the quadratic form $m_1^2+m_2^2+m_3^2-n_1^2-n_2^2-n_3^2-l=0$. This counting problem is executed by appealing to Proposition ~\ref{propmain}. 

Recall that an integer $n$ is a fundamental discriminant if, $n \equiv 1 \pmod{4}$ and square-free, or $n=4m$ with $m$ square-free and $m\equiv 2$ or $3 \pmod{4}$. 

To handle the $2$-adic congruence conditions, we set up some notation. Let $S=\left\{1,4\right\}$. To each $s \in S$ we associate certain residue classes in $\mathbf{Z}/4\mathbf{Z}$, or $\mathbf{Z}/8\mathbf{Z}$. Set $R(1) = \left\{5 \right\}$, $M(1)=8$, $R(4) = \left\{2,3 \right\}$, $M(4)=4$, and attach weights, $\tau(1,1) = 1, \tau(1,4) = \tau(4,1) = 2$ and $\tau(4,4) = 4$, to pairs $(s,t) \in S\times S$.

Also, for a positive integer $A$ define
$$ \rho(A) = \sum_{\substack{\bs{x} \in \left(\mathbf{Z}/A\mathbf{Z}\right)^3 \\ F(\bs{x}) \equiv 0 \mod{A}}} 1.$$Excluding fundamental discriminants that are congruent to $1 \pmod{8}$ in ~\eqref{eq:smoothsumdefinition}, we get by ~\eqref{eq:r3class} that
\begin{equation}\label{eq:ti}
\begin{split}
\widetilde{D}(X,l) &= \frac{1}{576}\sum_{(s,t) \in S\times S}\tau(s,t) \times \\&\quad \quad \sum_{\substack{sm-tn = l \\ -m \in R(s) \mod{M(s)}\\ -n \in R(t)  \mod{M(t)}}} \mu^2(m)\mu^2(n)r_3(m)r_3(n)\alpha\left(\frac{sm}{X+l}\right)\beta\left(\frac{tn}{X}\right)  \\
&=\sum_{(s,t)\in S\times S} T(s,t), 
\end{split}
\end{equation}
say.
For the rest of the proof we use boldface $\boldsymbol{x}$ to denote a 3-tuple $(x_1,x_2,x_3)$, and by $F(\boldsymbol{x})=|\boldsymbol{x}|_2^2$ we denote the square of the $L^2$ norm of $\boldsymbol{x}$. 
We detect the square-free condition in ~\eqref{eq:ti} by using the identity $\mu^2(n) = \textstyle\sum_{d^2 \mid n} \mu(d)$. For instance, we have
\begin{equation*}
T(1,1) = \frac{1}{576}\sum_{k=1}^{\infty} \mu(k) \sum_{\substack{m-n=l \\ n \equiv 0 \mod{k^2} \\ n \equiv -5 \mod{8} \\ m \equiv -5 \mod{8}}} r_3(n) \mu^2(m)r_3(m)\alpha\left(\frac{m}{X+l}\right)\beta\left(\frac{n}{X}\right).
\end{equation*}
In the following lemma we show that the $k$-sum can be truncated, and that the tail makes a small contribution. Define  
\begin{equation}\label{eq:weighht}
w(\boldsymbol{x},\boldsymbol{y})=\alpha(|\boldsymbol{x}|_2^2)\beta(|\boldsymbol{y}|_2^2). 
\end{equation}
Let $s \in S$. For an integer $j$ define the set 
\begin{equation*}
\mathcal{A}_j(s) = \left\{\boldsymbol{a} \in (\mathbf{Z}/M(s)j^2\mathbf{Z})^3 : \begin{split} & F(\boldsymbol{a}) \equiv 0 \pmod{j^2} \\ & F(\boldsymbol{a}) \in R(s) \mod{M(s)} \end{split}\right\}.
\end{equation*}

\begin{lemma}\label{truncationlemma}
Fix $\eta > 0$. Then for all $\ve > 0$ we have
\begin{equation*}
\begin{split}
\widetilde{D}(X,l) &= \frac{1}{576}\sum_{(s,t)\in S\times S}\tau(s,t)\sum_{\substack{j,k \leq X^{\eta} \\ (2,jk)=1,(j,k)^2 \mid l}} \mu(j) \mu(k) \times \\ 
&\quad \quad \sum_{\substack{\boldsymbol{a}_1 \in \mathcal{A}_j(s) \\ \bs{a}_2 \in \mathcal{A}_k(t) }} \sum_{\substack{\boldsymbol{m},\boldsymbol{n} \in \mathbf{Z}^3 \\ \boldsymbol{m} \equiv \boldsymbol{a} \mod{M(s)j^2} \\ \boldsymbol{n} \equiv \boldsymbol{b} \mod{M(t)k^2} \\ sF(\boldsymbol{m}) - tF(\boldsymbol{n}) = l}} w\left(\frac{\sqrt{s}\boldsymbol{m}}{X+l},\frac{\sqrt{t}\boldsymbol{n}}{X}\right) \\ 
& \quad  \quad \quad + O(X^{3/2-\eta}(X+l)^{1/2+\varepsilon}).
\end{split}
\end{equation*}
The implied constant depends only on $\varepsilon$. 
\end{lemma}
\begin{proof}
To simplify notation, we work with $T(1,1)$. The other terms are handled in exactly the same way. Opening up $\mu^2(n)$ we see that
\begin{equation*}
\begin{split}
T(1,1) &= \frac{1}{576}\sum_{k=1}^{\infty} \mu(k) \sum_{\substack{m-n=l \\ n \equiv 0 \mod{k^2} \\ n \equiv -5 \mod{8} \\ m \equiv -5 \mod{8}}} r_3(n) \mu^2(m)r_3(m)\alpha\left(\frac{m}{X+l}\right)\beta\left(\frac{n}{X}\right) \\
&= \sum_{k \leq X^{\eta}} + \sum_{k > X^{\eta}} = S_1 + S_2, \mbox{say.}
\end{split}
\end{equation*}
We have
\begin{equation*}
\begin{split}
S_2 &= \frac{1}{576}\sum_{k > X^{\eta}} \sum_{\substack{\boldsymbol{a}_2 \in \mathcal{A}_k(1)}} \sum_{\substack{\boldsymbol{m},\boldsymbol{n} \in \mathbf{Z}^3 \\ F(\boldsymbol{m}) \equiv -5 \mod{8} \\ \boldsymbol{n} \equiv \boldsymbol{a}_2 \mod{8k^2} \\ F(\boldsymbol{m}) - F(\boldsymbol{n}) = l}}\mu^2(F(\boldsymbol{m}))w\left(\frac{\boldsymbol{m}}{X+l},\frac{\boldsymbol{n}}{X}\right)\\
\end{split}
\end{equation*}
By choice of our weight function, $|\bs{m}| \ll (X+l)^{\frac{1}{2}}$ and $|\bs{n}| \ll X^{\frac{1}{2}}$. Furthermore, for each fixed $\bs{n}$, the number of $\bs{m}$ such that $F(\bs{m})-F(\bs{n}) - l =0$ is $O\left((X+l)^{\frac{1}{2}+\ve}\right)$. Since $k \ll X^{\frac{1}{2}}$ we have,
\begin{equation*}
\begin{split}
S_2 &\ll (X+l)^{\frac{1}{2}} \sum_{k > X^{\eta}} \sum_{\bs{a}_2 \in \mathcal{A}_k(1)} \sum_{\bs{n} \equiv \bs{a}_2 \pmod{8k^2}} 1 \\ 
&\ll (X+l)^{\frac{1}{2}} \sum_{k > X^{\eta}} \sum_{k^2 \mid n}r_3(n) 
\ll X^{3/2-\eta}(X+l)^{1/2+\varepsilon}.
\end{split}
\end{equation*}
We repeat this process by opening up $\mu^2(m)$ in $S_1$ to complete the proof.
\end{proof}

\begin{lemma}\label{singseriest}
Let $(s,t) \in S\times S$. Define the sum
\begin{equation*}\label{eq:tqjkl}
T_q(j,k,(s,t);l) = \sum_{\substack{\boldsymbol{a} \in \mathcal{A}_j(s) \\ \boldsymbol{b} \in \mathcal{A}_k(t)}}\sum_{\substack{(d,q)=1 \\ \boldsymbol{x} \mod{M(s)qj^2} \\ \boldsymbol{x} \equiv \boldsymbol{a} \mod{M(s)j^2} \\ \boldsymbol{y} \mod{M(t)qk^2} \\ \boldsymbol{y} \equiv \boldsymbol{b} \mod{M(t)k^2}}}e_q\left(d(F(\boldsymbol{x})-F(\boldsymbol{y})-l)\right).
\end{equation*}
Let $(2,jk)=1$. For $p$ a prime, let $j_p = v_p(j)$ and $k_p = v_p(k)$ be the $p$-adic valuations of $j$ and $k$ respectively. 
Then
\begin{equation*}
\sum_{q \leq Z}\frac{T_q(j,k,(s,t);l)}{q^6} = \gamma^{}(j,k,(s,t);l) + O(\rho(j^2)\rho(k^2)(jk)^{24}Z^{-\frac{3-\delta}{2}+\varepsilon}),
\end{equation*}
where $\gamma(j,k,(s,t);l) = (M(s)j^2)^3(M(t)k^2)^3 \gamma_2^{}(j,k,(s,t);l)\prod_{2 < p < \infty} \gamma_p(j,k;l)$ and
\begin{equation*}
\gamma_2^{}((s,t);l) = \lim_{t \to \infty} \frac{\# \left\{\begin{split}&\boldsymbol{x} \mod{2^{t}} \\ &\boldsymbol{y} \mod{2^{t}} \end{split}: \begin{split}& 2^t \mid sF(\boldsymbol{x})-tF(\boldsymbol{y})-l \\ & F(\boldsymbol{x}) \in R(s) \pmod{M(s)}\\ & F(\boldsymbol{y}) \in R(t) \pmod{M(t)}\end{split} \right\}}{2^{5t}}.
\end{equation*}
and
\begin{equation*}
\gamma_p^{}(j,k;l) = \lim_{t \to \infty} \frac{\# \left\{\begin{split}&\boldsymbol{x} \mod{p^{t}} \\ &\boldsymbol{y} \mod{p^{t}} \end{split}: \begin{split}& p^t \mid sF(\boldsymbol{x})-tF(\boldsymbol{y})-l \\ & p^{2j_p} \mid F(\boldsymbol{x}), p^{2k_p} \mid F(\boldsymbol{y}) \end{split} \right\}}{p^{5t}}.
\end{equation*}
when $p$ is odd.
\end{lemma}
\begin{proof}

To prove the lemma, we make the following claim, which is immediate from Lemma ~\ref{multlem}.
\begin{claim}
If $q=q_1q_2$, $j=j_1j_2$ and $k=k_1k_2$ with $(M(s)q_1j_1,q_2j_2) =1 $ and \newline $(M(t)q_1k_1,q_2k_2)=1$, then $$T_q(j,k,(s,t);l) = T_{q_1}(j_1,k_1,(s,t);l)T_{q_2}(j_2,k_2,(s,t);l).$$ 
\end{claim}
As a result, if $q=\prod p^{q_p}$, $j = \prod p^{j_p}$ and $k= \prod p^{k_p}$, then we have
$$T_q(j,k;l) = \prod_p T_{p^{q_p}}(p^{j_p},p^{k_p},(s,t);l).$$ By Lemma ~\ref{lemma28} we get that
\begin{equation*}
\begin{split}
\sum_{q \leq Z} \frac{T_q(j,k,(s,t);l)}{q^6} &= \sum_{q=1}^{\infty}\frac{T_q(j,k,(s,t);l)}{q^6} + O(\rho(j^2)\rho(k^2)(jk)^{24}Z^{-\frac{3-\delta}{2}+\varepsilon}).
\end{split}
\end{equation*}
Therefore, 
\begin{equation*}
\begin{split}
\sum_{q \leq Z} \frac{T_q(j,k,(s,t);l)}{q^6} &= \prod_p \sum_{k=0}^{\infty} \frac{T_{p^k}(p^{j_p},p^{k_p},(s,t);l)}{p^{6k}} \\ &\quad \quad + O(\rho(j^2)\rho(k^2)(jk)^{24}Z^{-\frac{3-\delta}{2}+\varepsilon}). \\
\end{split}
\end{equation*}
By a standard argument (for example, ~\cite[Lemma 2.2]{HB2}) it follows that 
\begin{equation*}
\sum_{q \leq Z} \frac{T_q(j,k,(s,t);l)}{q^6} = \gamma^{}(j,k,(s,t);l) + O(\rho(j^2)\rho(k^2)(jk)^{24}Z^{-\frac{3-\delta}{2}+\varepsilon}).
\end{equation*}
This completes the proof of the lemma. 
\end{proof}
\subsubsection{Applying the main proposition}

We apply Proposition ~\ref{propmain} and Lemma ~\ref{infty} to each of the terms that appear in Lemma ~\ref{truncationlemma}, with $Q_1(\bs{x}) = Q_2(\bs{x}) = x_1^2+x_2^2+x_3^2$, and $A_1 = M(s)j^2, A_2=M(t)k^2$ and $Y=X+l$. Observe that our weight function in ~\eqref{eq:weighht} satisfies $\|w\|_{N,1} \ll_N P^N \asymp \|w\|_{1,1}^N$. Moreover, from the nature of the function $w$, we can take $\delta = P^{-\frac{1}{2}}$ in Lemma ~\ref{infty}. Putting everything together, and using the fact that $\rho(j^2) \leq j^6$, we get that
\begin{equation*}\label{eq:prefinal}
\begin{split}
\widetilde{D}(X,l) &= \frac{X^{\frac{3}{2}}Y^{\frac{1}{2}}}{576}\sum_{(s,t)\in S\times S}\tau(s,t) \sigma_{\infty}(w,(s,t),l)\times \\ &\quad \quad \sum_{\substack{j,k \leq X^{\eta} \\(2,jk)=1 \\ (j,k)^2 \mid l}} \frac{\mu(j) \mu(k)\gamma^{}(j,k,(s,t);l)}{M(s)^3j^6M(t)^3k^6} \,  + \\
&\quad \quad \quad O_{\ve}\left(X^{\frac{3}{2}+38\eta}Y^{\frac{-1+\delta}{4}+\ve} + P^6X^{38\eta}Y^{\frac{5+\delta}{4}+\ve} + X^{\frac{3}{2}-\eta}Y^{\frac{1}{2}+\ve}\right),
\end{split}
\end{equation*}
where \begin{equation*}
\begin{split}
\sigma_{\infty}(w,(s,t),l) &= \lim_{\kappa \to 0} \frac{1}{2\kappa} \int_{|sF(\boldsymbol{x})-t\frac{XF(\boldsymbol{y})+l}{Q^2}|\leq \kappa}w((\sqrt{s}\bs{x}_1,\sqrt{t}\bs{x}_2)) \, d\boldsymbol{x} \\
&= \frac{\sigma_{\infty}(w,(1,1),l)}{\tau^3(s,t)}
\end{split}
\end{equation*} is the singular integral that corresponds to $T(s,t)$ in Lemma ~\ref{infty}. The first error term above comes from applying Lemma ~\ref{singseriest}, the second from the application of Proposition ~\ref{propmain}, and the last error term results from invoking Lemma ~\ref{truncationlemma}. It is easy to see that $\frac{\gamma^{}(j,k,(s,t);l)}{M(s)^3j^6M(t)^3k^6} \ll (jk)^{-\frac{3}{2}}$, so we may extend the $j$ and $k$ sums to $\infty$ to get
\begin{equation}\label{eq:thm1final}
\begin{split}
\widetilde{D}(X,l) &= \frac{\widehat{\sigma}_w(l)}{576}X^{\frac{3}{2}}Y^{\frac{1}{2}} \\
&\quad \quad + O\left(X^{\frac{3}{2}+38\eta}Y^{\frac{-1+\delta}{4}+\ve} + P^6X^{38\eta}Y^{\frac{5+\delta}{4}+\ve} + X^{\frac{3}{2}-\eta}Y^{\frac{1}{2}+\ve}\right),
\end{split}
\end{equation}
where 
\begin{equation}\label{eq:const12}
\widehat{\sigma}_w(l) = \sigma_{\infty}(w,(1,1),l) \prod_{p < \infty} \sigma_p(l),
\end{equation}
\begin{equation*}\label{eq:constant1}
\sigma_p(l) = \begin{cases} \sum_{(s,t) \in S\times S}\frac{\gamma_2^{}((s,t),l)}{\tau(s,t)^2} &\mbox{$p=2$} \\ 
\gamma_p(1,1;l) - \gamma_p(p,1;l)-\gamma_p(1,p;l) + \gamma_p(p,p;l) &\mbox{$2 < p < \infty$}.
\end{cases}
\end{equation*}
\subsubsection{Removing the weight $w$}\label{removingweightssection}
By our choice of test function $w$ it follows that
\begin{equation*}\label{eq:removeweight}
\begin{split}
\sigma_{\infty}(w,(1,1),l) &= \sigma_{\infty}(l) + O(1/P),
\end{split}
\end{equation*}
$\sigma_{\infty}(l) = \lim_{\kappa \to 0} \frac{1}{2\kappa} \int_{|F(\bs{x}_1) - \frac{XF(\bs{x}_2)+l}{X+l}|\leq \kappa} d\bs{x}$ and the integral is over the region $$\mathcal{R} = \left\{\bs{x} \in \mathbf{R}^6 : |F(\bs{x}_1)| \leq 1, |F(\bs{x}_2)| \leq 1\right\}.$$ Therefore, by taking $X^{\eta} = P = X^{\frac{1}{30}}Y^{-\frac{3+\delta}{180}-2\ve}$, it follows from Lemma ~\ref{smoothinglemma} and ~\eqref{eq:thm1final} that  
\begin{equation*}
D(X,l) = \frac{\widehat{\sigma}(l)}{576} X^{\frac{3}{2}}Y^{\frac{1}{2}} + O_{\ve}\left(X^{\frac{3}{2}-\frac{1}{30}}Y^{\frac{1}{2}+\frac{3+\delta}{180}+\ve}\right),
\end{equation*}
where
\begin{equation}\label{eq:const11}
\widehat{\sigma}(l) = \prod_{p \leq \infty}\sigma_p(l).
\end{equation}
This completes the proof of Theorem ~\ref{thm1}.

\begin{remark}\label{remark}
It is easy to explicitly compute the singular integral. Indeed, we have for $l \neq 0$ that
\begin{equation}\label{eq:explicit}
\sigma_{\infty}(l) = \frac{\pi^2}{3X^{\frac{3}{2}}(X+l)^{\frac{1}{2}}}\left\{(2X+l)\sqrt{X(X+l)} - l^2 \arcsinh\left(\sqrt{\frac{X}{l}}\right)\right\}. 
\end{equation}
To see this, recall that
\begin{equation*}
\sigma_{\infty}(l) = \lim_{\kappa \to 0} \frac{1}{2\kappa} \int_{|F(\bs{x}_1) - \frac{XF(\bs{x}_2)+l}{X+l}|\leq \kappa} d\bs{x}.
\end{equation*}
Integrating first over $\bs{x}_1$ we have that
\begin{equation*}
\begin{split}
\sigma_{\infty}(l) &= \lim_{\kappa \to 0} \frac{1}{2\kappa} \frac{4\pi}{3}\int \left(\frac{XF(\bs{x}_2)+l}{X+l}+\kappa\right)^{\frac{3}{2}} - \left(\frac{XF(\bs{x}_2)+l}{X+l}-\kappa\right)^{\frac{3}{2}} \, d\bs{x}_2 \\ 
&= \frac{4\pi}{3}\frac{1}{\sqrt{X+l}}\int_{F(\bs{x}_2)\leq 1} \sqrt{XF(\bs{x}_2)+l} \, d\bs{x}_2 \\  
\end{split}
\end{equation*}
Switching to spherical co-ordinates, we find that
\begin{equation}\label{eq:1}
\begin{split}
\sigma_{\infty}(l) &= \frac{16\pi^2}{3\sqrt{X+l}}\int_{0}^1 r^2\sqrt{r^2X+l} \, dr,
\end{split}
\end{equation}
and ~\eqref{eq:explicit} follows. As a result, we see that $\sigma_{\infty}(l) = \frac{4\pi^2}{3} + O(X^{-\ve})$ whenever $l \ll X^{1-2\ve}$. Denote by $I$ the integral over $r$ in ~\eqref{eq:1}. Set $r^2X=t$. Then we have that
\begin{equation*}
\begin{split}
I &= \frac{1}{2X^{\frac{3}{2}}}\int_{0}^X t^{\frac{1}{2}}(t+l)^{\frac{1}{2}} \, dt \\
&= \frac{\sqrt{X+l}}{3} - \frac{1}{6X^{\frac{3}{2}}} \int_{0}^X t^{\frac{3}{2}}(t+l)^{-\frac{1}{2}} \, dt.
\end{split}
\end{equation*}
As a result, when $l \gg X^{1+2\ve}$ we find that $\sigma_{\infty}(l) = \frac{16\pi^2}{9} + O(X^{-\ve})$. Moreover, in the range $0 \leq l \ll X^{2-2\ve}$ we have that $1 \ll \sigma_{\infty}(l) \ll 1$, and the implied constants are absolute. 
\end{remark}
\subsection{Proof of theorem ~\ref{thm3}}
The proof of Theorem ~\ref{thm3} is similar to the proof of Theorem ~\ref{thm1}, so we only give a brief outline. Here we adopt the notation where a $4$-tuple $\bs{x}$ is written $\bs{x} = (\bs{x}_1,\bs{x}_2)$, and $\bs{x}_1$ is a 3-tuple. Once again it suffices to consider the following weighted analogue of $S(X,d)$,
\begin{equation}\label{eq:swd}
\tilde{S}(X,d) = \sideset{}{^{\flat}}\sum \beta(n/X)h(-(n^2+d)).
\end{equation} 
Let $Q_1(\bs{x}) = x_1^2+x_2^2+x_3^2$ and $Q_2(x) = x^2$.  

As before, we need some notation to handle the $2$-adic congruence conditions. Let $S = \left\{3,4,8\right\}.$ Let $M(4)=M(8)=16$ and $M(3)=8$. Let $\tau(4)=\tau(8)=2$ and $\tau(3) = 1$. For $s \in S$ define  
\begin{equation*}
\mathcal{A}_j(s) = \left\{\bs{a}_1 \in \left(\mathbf{Z}/M(s)j^2\mathbf{Z}\right)^3 : \begin{split} & Q_1(\bs{a}_1) \equiv s \mod{M(s)} \\ &Q_1(\bs{a}_1) \equiv 0 \mod{j^2} \end{split} \right\}.
\end{equation*}
Since we are excluding fundamental discriminants $-(n^2+d)$ that are congruent to $1 \pmod{8}$ we get, for $\eta > 0$ and any $\ve > 0$ that
\begin{equation*}
\begin{split}
\widetilde{S}(X,d) &= \frac{1}{24}\sum_{s \in S} \tau(s) \sum_{\substack{j \leq X^{\eta}\\(2,j)=1}} \mu(j) \sum_{\bs{a}_1 \in \mathcal{A}_j(s)} \sum_{\substack{\bs{x}_1 \equiv \bs{a}_1 \mod{M(s)j^2} \\ Q_1(\bs{x}_1) - Q_2(\bs{x}_2) = d}} \alpha\left(\frac{Q_1(\bs{x}_1)}{X^2+d}\right)\beta\left(\frac{\bs{x}_2}{X}\right) \\ 
&\quad \quad + O_{\ve}(X^{1-\eta}(X^2+d)^{\frac{1}{2}+\ve}).
\end{split}
\end{equation*}
Let \begin{equation*}
\gamma_2(s,d) =  \lim_{t \to \infty} \frac{\# \left\{\begin{split}&\boldsymbol{x} \mod{2^{t}} \end{split}: \begin{split}& 2^t \mid Q_1(\boldsymbol{x}_1)-Q_2(\boldsymbol{x_2})-d \\ & Q_1(\bs{x}_1) \equiv s \mod{M(s)} \end{split} \right\}}{2^{3t}}
\end{equation*} 
and set
$\gamma(d) = \prod_{p < \infty} \gamma_p(j)$, where
$$\gamma_2(d) = \sum_{s \in S} \tau(s)\gamma_2(s,d),$$ and 
\begin{equation*}
\gamma_p(j;d) = \lim_{t \to \infty} \frac{\# \left\{\begin{split}&\boldsymbol{x} \mod{p^{t}} \end{split}: \begin{split}& p^t \mid Q_1(\boldsymbol{x}_1)-Q_2(\boldsymbol{x_2})-d \\ & p^{2v_p(j)} \mid Q_1(\boldsymbol{x}_1) \end{split} \right\}}{p^{3t}}
\end{equation*}
when $p$ is odd. Also define
\begin{equation*}
\gamma_{\infty}(d) = \lim_{\kappa \to 0} \frac{1}{2\kappa} \int_{|Q_1(\bs{x}_1) - \frac{X^2Q_2(\bs{x}_2)+d}{X^2+d}|\leq \kappa} d\bs{x},
\end{equation*}
where the integral on the right is over the region $$\mathcal{R} = \left\{\bs{x} \in \mathbf{R}^4 : |Q_1(\bs{x}_1)| \leq 1, |Q_2(\bs{x}_2)| \leq 1\right\}.$$ Following the proof of Theorem ~\ref{thm1}, and replacing $X$ by $X^2$, and $Y$ by $(X^2+d)$ we get that
\begin{equation}
\begin{split}
~\widetilde{S}(X,d) &= \frac{\widetilde{\sigma}(d)}{24}X(X^2+d)^{\frac{1}{2}} \, + \\ &\quad \quad O_{\ve}\left(P^{4}X^{13\eta}(X^2+d)^{\frac{3}{4}+\ve} + X^{1-\eta}(X^2+d)^{\frac{1}{2}+\ve}\right),  
\end{split}
\end{equation}
where $\widetilde{\sigma}(d) = \prod_{p \leq \infty} \widetilde{\sigma}_p(d)$ and
\begin{equation}\label{eq:constant2}
\widetilde{\sigma}_p(d) = \begin{cases} \gamma_2(d) &\mbox{$p = 2$} \\ 
\gamma_p(1;d) - \gamma_p(p;d) &\mbox{$2 < p < \infty$} \\
\gamma_{\infty}(d) &\mbox{$p = \infty$}. \end{cases}
\end{equation}
To complete the proof, take $X^{\eta} = P = X^{\frac{1}{18}}(X^2+d)^{-\frac{1}{72}-2\ve}$ to get the desired estimate for $S(X,d)$, since $S(X,d) - \tilde{S}(X,d) \ll_{\ve} X(X^2+d)^{\frac{1}{2}+\ve}/P$. 
\subsection{Proof of Theorems ~\ref{iwaniec} and ~\ref{thmr}}
Let $Q(x_1,x_2) = x_1^2+x_2^2$. To prove Theorem ~\ref{iwaniec}, we start with the smoothed sum $S = \textstyle\sum_{m-n=l}\alpha\left(\frac{m}{X+l}\right)\beta\left(\frac{n}{X}\right)r(m)r(n)$, and we see that it differs from the unsmoothed sum by at most $O(X^{1+\ve}/P)$. Applying Proposition ~\ref{propmain} with $Q_1 = Q_2 = Q$, $A_1 = A_2 = 1$ and \newline $w(\bs{x}) = \alpha(Q_1(\bs{x}_1))\beta(Q_2(\bs{x}_2),$ we get
\begin{equation*}
S = c'(l)X + O_{\ve}(P^4(X+l)^{\frac{3}{4}+\ve}),
\end{equation*}
where $c'(l)=c_{\infty}(w,l)\prod_{p}c_p(l)$, with $c_{\infty}(w,l)$ and $c_p(l)$ are as in ~\eqref{eq:cwl} and ~\eqref{eq:ccp}. As before, we have
\begin{equation*}
\begin{split}
c_{\infty}(w,l) &=  \lim_{\kappa \to 0} \frac{1}{2\kappa} \int_{|Q(\bs{x}_1) - \frac{XQ(\bs{x}_2)+l}{X+l}|\leq \kappa} d\bs{x} + O(1/P),
\end{split}
\end{equation*}
where we integrate over the region $$\mathcal{R} = \left\{\bs{x} \in \mathbf{R}^4 : |Q(\bs{x}_1)| \leq 1, |Q(\bs{x}_2)| \leq 1\right\}.$$ Since we are integrating over discs in $\mathbf{R}^2$, it is easy to see that $c_{\infty}(w,l) = \pi^2 + O(1/P)$. Setting $c(l) = \pi^2 \prod_{p < \infty}c_p(l)$ and $P = X^{\frac{1}{5}}(X+l)^{-\frac{3}{20}-2\ve}$ we get that 
\begin{equation*}
\sum_{n \leq X} r(n)r(n+l) = c(l)X + O_{\ve}\left(X^{\frac{4}{5}}(X+l)^{\frac{3}{20}+\ve}\right).
\end{equation*}
It is well-known that $c(l) \ll l^{\ve}$, and $c(l) \neq 0$ if and only if $c_p(l) \neq 0$, if and only if the equation $x_1^2+x_2^2-x_3^2-x_4^2-l=0$ has a solution in $\mathbf{Q}_p$. 

The proof of Theorem ~\ref{thmr} is similar, and follows at once from Proposition ~\ref{propmain} by taking $A_1=A_2=1$, and by setting $w(\bs{x}) = \alpha(Q_1(\bs{x}_1))\beta(Q_2(\bs{x}_2)),$ with $P = X^{\frac{m}{2(2m+1)}}(X+l)^{-\frac{3+\delta}{4(2m+1)}}.$ 

\begin{ack}
I would like to thank my supervisor, Tim Browning, for suggesting this problem to me, and for his guidance throughout the process of writing this paper, including his detailed comments on earlier drafts. I would also like to thank Jonathan Bober and Rainer Dietmann for their comments and suggestions. 
\end{ack}

\end{document}